\documentclass{amsart}
\usepackage{subfiles}
\usepackage{graphicx} 
\usepackage{geometry}
\usepackage{mathtools}
\usepackage{amsthm,amssymb,amsmath}
\usepackage{ytableau}
\usepackage{tikz}
\usepackage{amsmath}
\usepackage{enumitem}
\usepackage{nicematrix}
\usepackage{empheq}
\usepackage[marginal,multiple]{footmisc}
\usepackage{tikz-cd}
\usepackage{booktabs}
\usepackage{soul}
\usepackage{comment}
\usepackage{hyperref}

\usepackage[most]{tcolorbox}

\usepackage{biblatex}
\newtheorem{proposition}{Proposition}[section]
\newtheorem{theorem}{Theorem}[section]
\newtheorem{lemma}{Lemma}[section]
\newtheorem{corollary}{Corollary}[section]

\theoremstyle{definition}
\newtheorem{definition}{Definition}[section]
\newtheorem{example}{Example}[section]
\newtheorem{remark}{Remark}[section]
\newtheorem{problem}{Problem}[section]

\newcommand{\bC}{\mathbb{C}}
\newcommand{\Fl}{\mathrm{Fl}}
\newcommand{\Gr}{\mathrm{Gr}}
\newcommand{\bP}{\mathbb{P}}
\newcommand{\bZ}{\mathbb{Z}}
\newcommand{\Poin}{\mathrm{P}}
\newcommand{\vspan}{\mathrm{span}}
\newcommand{\im}{\mathrm{im}}

\newcommand{\multibinom}[2]{
  \left[\genfrac{}{}{0pt}{}{#1}{#2}\right]
}

\newcommand{\SG}[1]{{\color{blue}#1--SG}}

\title{Irreducible components of two-column $\Delta$-Springer fibers}

\author{Joshua P.~Connor}
\address{UC Davis Library, University of California Davis\\ 100 NW Quad, Davis CA 95616, USA}
\email{jconnor@ucdavis.edu}

\author{Sean T.~Griffin}
\address{Faculty of Mathematics, University of Vienna, Oskar-Morgenstern-Platz 1, 1090 Vienna, Austria}
\email{sean.griffin@univie.ac.at}
\thanks{Sean T.~Griffin was supported by ERC grant “Refined invariants in combinatorics, low-dimensional topology and geometry of moduli spaces” No.~101001159}

\author{Kavish A.~Purohit}
\address{Department of Mathematics, University of California Davis\\ One Shields Ave, Davis CA 95616, USA}
\email{kavishpurohit20@gmail.com}

\date{\today}

\begin{document}

\begin{abstract}
     The $\Delta$-Springer fibers $Y_{n,\lambda,s}$, introduced by Levinson, Woo, and the second author, generalize Springer fibers for $\mathrm{GL}_n(\bC)$ and give a geometric interpretation of the of the Delta Conjecture from algebraic combinatorics (at $t=0$). We prove that all irreducible components of the $\Delta$-Springer fiber $Y_{n,n-1}=Y_{n,(1^{n-1}),n-1}$ are smooth. In fact, we prove that any intersection of irreducible components of $Y_{n,n-1}$ is a smooth Hessenberg variety which has the structure of an iterated Grassmannian fiber bundle. We then give a presentation of the singular cohomology ring of each irreducible component of $Y_{n,n-1}$ and a combinatorial formula for the Poincar\'e polynomial of an arbitrary union of intersections of irreducible components in terms of arm and leg statistics on Dyck paths.
\end{abstract}

\maketitle

\section{Introduction}

Springer fibers, originally defined by T.A.~Springer~\cite{Springer-WeylGrpReps,Springer-TrigSum}, are subvarieties of the complete flag variety $\mathrm{Fl}(1^n)$ that have important applications to the representation theory of Lie algebras and Weyl groups. Springer fibers for $\mathrm{GL}_n\mathbb{(C)}$ are indexed by partitions $\lambda$ of $n$, and denoted by $\mathcal{B}_\lambda$. Precisely, given a nilpotent $n\times n$ matrix $x$ whose Jordan form has block sizes recorded by the partition $\lambda = (\lambda_1,\lambda_2,\dots, \lambda_m)$ of $n$, 
\[
\mathcal{B}_\lambda = \{V_\bullet\in \mathrm{Fl}(1^n)\mid xV_i\subseteq V_i\text{ for all }i\}.
\]
One remarkable fact about these varieties is that their irreducible components are in bijection with standard Young tableaux of shape $\lambda$~\cite{Spaltenstein}.

Understanding the geometric properties of the irreducible components themselves has generated independent interest. So far, only partial results have been proven in this direction. For example, in the case of $\lambda$ two rows, it is known that all irreducible components of $\mathcal{B}_\lambda$ are smooth. In fact, in this case all irreducible components and intersections of irreducible components are iterated projective bundles~\cite{fung2002topology}. On the other hand, when $\lambda$ has two columns it is known that not all irreducible components are smooth, but combinatorial conditions are known for when they are smooth iterated Grassmannian bundles~\cite{fresse2009singular,fresse-melnikov}. 

More recently, Levinson, Woo, and the second author have defined the \textbf{$\Delta$-Springer fibers} which generalize Springer fibers~\cite{GLW}. Their motivation was to connect the Delta Conjecture from algebraic combinatorics~\cite{HRW} to the geometry of Springer fibers. Given two positive integers $k\leq n$, let $\mathrm{Fl}(1^n,(k-1)(n-k))$ be the partial flag variety whose successive quotients have dimensions $(1,1,\dots, 1, (n-k)(k-1))$, and let $x$ be a nilpotent matrix whose Jordan type is the partition $(n-k+1)^{k-1} = (n-k+1,n-k+1,\dots, n-k+1)$. Then the $\Delta$-Springer fiber $Y_{n,k}$ is defined as 
\begin{equation}\label{eq:YnkDef}
Y_{n,k} \coloneqq \{V_\bullet\in \mathrm{Fl}(1^n,(n-k)(k-1))\mid xV_i\subseteq V_i, \ker(x)\subseteq V_n\}.
\end{equation}
In fact, this definition extends to a class of varieties $Y_{n,\lambda,s}$ indexed by triples $n,\lambda,s$ where $\lambda$ is a partition with size $|\lambda|\leq n$ and $s$ is an integer greater than or equal to the length of $\lambda$, $s\geq \ell(\lambda)$. In the case when $\lambda$ is a single column $\lambda=(1^k)$ and $s=k$, they specialize to the $\Delta$-Springer fibers defined above, $Y_{n,k} = Y_{n,(1^k),k}$. Levinson, Woo, and the second author showed that the symmetric group $S_n$ acts on $H^*(Y_{n,\lambda,s})$ and gave a presentation for this ring. In the special case $Y_{n,k}$, the ring $H^*(Y_{n,k})$ is isomorphic to the \emph{generalized coinvariant ring} $R_{n,k}$ of Haglund, Rhoades, and Shimozono~\cite{HRS1} whose graded $S_n$-representation structure encodes the symmetric function in the Delta Conjecture (at $t=0$). Furthermore, Gillespie and the second author~\cite{GG} have realized $Y_{n,k}$ (and $Y_{n,\lambda,s}$ more generally) as an instance of a \emph{Borho--MacPherson variety} and used this to give another combinatorial formula for the graded $S_n$-representation $H^*(Y_{n,\lambda,s})$. As a corollary, they obtained a new Schur expansion for the Delta Conjecture symmetric function at $t=0$.

In \cite{Wilbert}, Lacabanne, Vaz, and Wilbert analyzed the irreducible components of $Y_{n,\lambda,s}$ in the case of two rows, meaning $Y_{n,\lambda,s}$ for $\ell(\lambda)\leq 2$. They proved that each irreducible component in this case is an iterated projective bundle. Furthermore, they proved that in the two-row case, each $\Delta$-Springer fiber is isomorphic to a particular subvariety of an \emph{exotic Springer fiber} of Kato~\cite{Kato}. They then gave a topological realization of $Y_{n,\lambda,s}$ in this case as a subspace of a product of 2-spheres and a skein-theoretic presentation of the cohomology ring $H^*(Y_{n,\lambda,s})$, which allowed them to realize the $S_n$ action topologically.

In this article, we analyze the irreducible components of the $\Delta$-Springer fibers when $\lambda = (1^{n-1})$ and $k=n-1$, meaning $Y_{n,n-1} = Y_{n,(1^{n-1}),n-1}$. We have the following theorem.
\begin{theorem}\label{thm:MainThmIntro}
     Each irreducible component of $Y_{n,n-1}$, and any nonempty intersection of irreducible components, is an iterated fiber bundle in which each fiber is a Grassmannian variety.
\end{theorem}
 Additionally, we give a Hilbert series formula for any union of intersections of components, as well as a quotient presentation for the cohomology of each irreducible component.

 \begin{theorem}
     The irreducible components of $Y_{n,n-1}$ are indexed by $K^2,\dots, K^n$. The integral singular cohomology ring of the $i^{th}$ component has the following presentation,
     \[
     H^*(K^i) \cong \mathbb{Z}[x_1,\dots, x_n]/I_n^i,
     \]
     where $I_n^i$ is the ideal generated by
     \begin{itemize}
         \item $e_2(x_1,\dots,x_n),\dots, e_n(x_1,\dots, x_n)$,
         \item $h_j(x_1,\dots, x_{i-1})$ for $j\geq n+1-i$
         \item $h_j(x_i,\dots, x_n)$ for $j\geq i-1$.
     \end{itemize}
 \end{theorem}
 Furthermore, we generalize some of these statements to the setting of $Y_{n,(1^{n-1}),s}$ for any $s\geq n-1$. 
 See Theorem~\ref{thm:GrassBundle}, Proposition~\ref{prop:ComponentContain} and Theorem~\ref{thm:QuotientRing}.

The organization of the paper is as follows. In Section~\ref{sec:Background}, we give relevant background. In Section~\ref{sec:Irred Cmpts}, we give our main result, that arbitrary intersections of irreducible components are iterated Grassmannian bundles. 
In Section~\ref{sec:Poincare}, we give formulas for the Poincar\'e polynomials of unions of intersections of irreducible components. 
In Section~\ref{sec:SingularCohomology}, we give a quotient ring presentation for the singular cohomology ring of an irreducible component. 
In Section~\ref{sec:Future}, we offer potential directions of expansion on the topic and list some open problems.

\subsection{Acknowledgments} 
We thank Andrey Pravdic and Arik Wilbert for helpful discussions.

\section{Background}\label{sec:Background}

\subsection{Tableaux}

A \textbf{partition} of a positive integer $n$ is a weakly decreasing tuple of integers $\lambda = (\lambda_1\geq \lambda_2\geq\cdots \geq \lambda_\ell>0)$ such that $\sum_i\lambda_i = n$. When writing partitions, we will often abbreviate the partition $(a,\dots,a)$ with $b$ many $a$'s by $(a^b)$.

The \textbf{Young diagram} of $\lambda$, in English notation, is formed by drawing $\lambda_i$ boxes or \textit{cells} in the $i^{th}$ row from the top. Given $\mu$ a partition whose Young diagram is contained in that of $\lambda$, we can form the \textbf{skew shape} $\lambda/\mu$ by deleting the left-most $\mu_i$ boxes from row $i$ of the Young diagram of $\lambda$ for each $i$. A \textbf{standard Young tableau} of skew shape $\lambda/\mu$ is a labeling of its cells with integers $1,2,\dots, n$, where $n$ is the number of boxes, that increases down each column and left to right across each row. For example, a standard Young tableau on the skew shape $(3,2,2,1)/(2,1)$ is 
\[
\begin{ytableau}
    \none & \none & 4 \\
    \none & 3\\
    1 &5\\
    2\\
\end{ytableau}.
\]
We denote by $\mathrm{SYT}(\lambda/\mu)$ the set of standard Young tableaux on $\lambda/\mu$.

\subsection{Flag varieties}

    Given a positive integer $n$ and a composition $\alpha = (\alpha_1,\dots, \alpha_\ell)$ of $m$, a \textbf{partial flag of type $\alpha$} is a chain of subspaces of $\bC^n$. A \textbf{flag} of $\mathbb{C}^n$ is a chain of subspaces $\{0\}=V_0 \subseteq V_1\subseteq V_2\subseteq \cdots \subseteq V_\ell = \bC^n$ such that $\dim(V_i/V_{i-1})=\alpha_i$ for all $i\geq 1$. 
    The associated \textbf{partial flag variety}, denoted by $\mathrm{Fl}(\alpha)$, is the set of all partial flags of type $\alpha$. In the case when $\alpha = (1,1,\dots, 1)=(1^n)$, we say $\Fl(1^n)$ is the \textbf{complete flag variety}. When $\alpha$ has two parts, $\alpha = (k,n-k)$, then the partial flag variety can be identified with the Grassmannian of $k$-dimensional subspaces of $\bC^n$, $\Fl(k,n-k)\cong \Gr(k,n)$.
    
    Given a permutation $w$, the \textbf{permutation flag} is the flag whose $i^{th}$ subspace is the span of the basis vectors corresponding to the first $i$ letters of the permutation $w$,
    \[
    V_i = \mathrm{span}\{e_{w_1},\dots, e_{w_i}\}.
    \]
    
    \bigbreak\noindent Given a nonzero vector $v\in \mathbb{C}^n$, it can be written uniquely in the standard basis $\{e_i\}$ as
    \[
    v = e_i + \sum_{j < i} a_j e_j
    \]
    for some $i$ and some $a_j\in \mathbb{C}$. Given a complete flag $V_\bullet\in \mathrm{Fl}(1^n)$, there is a unique permutation $w$ in the symmetric group $S_n$ and scalars $a_{i,j}$ such that if $V_i = \mathrm{span}\{v_1,\dots, v_i\}$ for all $i$,
    \begin{equation}\label{eq:spanSchubform}
    v_i = e_{w_i} + \sum_{j< w_i} a_{i,j} e_j,
    \end{equation}
    where $a_{i,j}=0$ unless $w^{-1}(j) > i$. For a fixed $w\in S_n$, the \textbf{Schubert cell} $X_w^\circ$ is the set of all flags $V_\bullet$ whose spanning vectors have the form \eqref{eq:spanSchubform}. Equivalently, we may put each flag in matrix coordinate form by recording the $i^{th}$ spanning vector $v_i$ as the $i^{th}$ column of a matrix.

    Let $n\leq K$ be positive integers. In this paper, we primarily focus on partial flags in the vector space $\mathbb{C}^K$ of type $(1,1,\dots,1, K-n) = (1^n, K-n)$. The Schubert cells for these partial flags are defined by partial permutations of the form $w = w_1w_2\dots w_n$ where $w_1,\dots, w_n$ are distinct elements of $\{1,\dots, K\}$. We may similarly find representative vectors $v_i$ of the parts of the flag $V_1,\dots, V_n$ as in \eqref{eq:spanSchubform}. The matrix coordinates of $V_\bullet$ are thus given by a $K\times n$ rectangular matrix with $v_i$ in the $i^{th}$ column. 

    \begin{example}
    for $n=4$, $K=6$ and $w=3614$, elements of $X_w^\circ$ are flags with spanning vectors of the form
    \begin{align*}
        v_1 &= e_3 + a_{1,1}e_1 + a_{1,2}e_2\\
        v_2 &= e_6 + a_{2,1}e_1 + a_{2,2}e_2 + a_{2,4}e_4 + a_{2,5}e_5\\
        v_3 &= e_1\\
        v_4 &= e_4
    \end{align*}
   and whose matrix coordinates are
    $$
    \begin{bmatrix}
        a_{1,1} & a_{2,1} & 1 & 0 \\
        a_{1,2} & a_{2,2} & 0 & 0 \\
        1 & 0 & 0 & 0 \\
        0 & a_{2,4} & 0 & 1 \\
        0 & a_{2,5} & 0 & 0 \\
        0 & 1 & 0 & 0
    \end{bmatrix}.
    $$
    \end{example}
    
\subsection{Springer Fibers}

Given a nilpotent $n\times n$ matrix $x$, the \textbf{Springer fiber} associated to $x$ is the set of all complete flags that are preserved by $x$,
\[
\mathcal{B}_x \coloneqq \{V_\bullet \in \mathrm{Fl}(1^n) \mid xV_i \subseteq V_i\text{ for all }i\}.
\]
Springer fibers have importance in geometric representation theory. Namely, T.A.~Springer used them to geometrically construct the irreducible representations of the symmetric group $S_n$~\cite{Springer-WeylGrpReps,Springer-TrigSum}.

Let $\lambda$ be the Jordan type of $x$, which is the partition recording the sizes of the Jordan blocks of the Jordan canonical form of $x$, then Springer fibers are alternatively indexed by the partition $\lambda$, $\mathcal{B}_\lambda\coloneqq \mathcal{B}_x$. By work of Spaltenstein~\cite{Spaltenstein}, the irreducible components of $\mathcal{B}_\lambda$ are in bijection with standard Young tableaux $\mathrm{SYT}(\lambda)$. Since $\mathcal{B}_x$ is connected and typically has multiple irreducible components, then $\mathcal{B}_\lambda$ is typically singular. 

Only partial progress has been made to characterize the smoothness of individual irreducible components. For example, Fung~\cite{fung2002topology} has proven that when $x$ has at most two Jordan blocks, or equivalently when the Jordan type $\lambda$ has at most 2 rows, then all irreducible components of $\mathcal{B}_x$ are smooth, and all intersections of components are smooth. Fresse has characterized the irreducible components of $\mathcal{B}_x$ that are smooth in the case when $\mu$ has at most 2 columns. Fresse and Melnikov~\cite{fresse-melnikov} then characterized the $\lambda$ for which all irreducible components of $\mathcal{B}_x$ are smooth.

\subsection{$\Delta$-Springer fibers}\label{subsec:Delta-Springer}

The \textbf{$\Delta$-Springer fibers} are a generalization of the type A Springer fibers introduced by the second author, Levinson, and Woo~\cite{GLW}. Given integers $k\leq n$, a partition $\lambda$ of $k$, and an integer $s\geq \ell(\lambda)$, let $x$ be a nilpotent matrix with Jordan type $(n-k,n-k,\dots, n-k)+ \lambda$ (where there are $s$ many $n-k$ and addition is component-wise). The \textbf{$\Delta$-Springer fiber} is
\[
Y_{n,\lambda,s} \coloneqq \{V_\bullet\in \Fl(1^n,(n-k)(s-1))\mid \mathrm{im}(x^{n-k}) \subseteq V_n, \, xV_i\subseteq V_i\}.
\]
Observe that when $n=k=s$, then $\lambda$ is a partition of $n$, and $Y_{n,\lambda,s}$ is the usual Springer fiber of complete flags associated with $\lambda$.

The original motivation behind the definition was to find a geometric realization of the Delta Conjecture. In forthcoming work, Gillespie, Gorsky and the second author realize this goal by connecting $\Delta$-Springer fibers to the combinatorics of the Delta Conjecture and Rational Shuffle Theorems~\cite{GGG2}.

In this article, we focus on a particular case of $\Delta$-Springer fibers: When $k=n-1$, and  $\lambda = (1^{n-1})$ a single column, then $x$ is a nilpotent of type $(2^{n-1},1^{s-n+1})$. In this case, $\mathrm{im}(x^{n-k}) = \im(x)$, and hence
\[
Y_{n,(1^{n-1}),s} = \{V_\bullet\in \Fl(1^n,s-1)\mid \im(x)\subseteq V_n,\, xV_i\subseteq V_i\}.
\]
Specializing even further, when $s=n-1$ then $\im(x)=\ker(x)$ and we can alternatively write
\[
Y_{n,n-1}=Y_{n,(1^{n-1}),n-1} = \{V_\bullet\in \Fl(1^n,n-2)\mid \ker(x)\subseteq V_n,\, xV_i\subseteq V_i\}.
\]

Throughout the article, we allow $s\geq n-1$ arbitrary and use the following specified nilpotent operator $x$. Let $x$ act on the basis elements $e_1,\dots, e_{n-1+s}$ of $\bC^{n-1+s}$ by the rule:
\[
xe_{i} = \begin{cases} 
e_{i-(n-1)} & \text{if }n\leq i \leq 2n-2,\\
0 & \text{otherwise.}
\end{cases}
\]

An affine paving of $Y_{n,\lambda,s}$ has been given in \cite{GLW} (see also \cite{griffin-HL}), and a description of the irreducible components as cell closures has been given in~\cite{GLW}. 

\begin{definition}\label{def:Cells}
    Given a standard Young tableau $T$ of skew shape $(2,1^{n-2})/(1^{s-n+1})$, let $i$ be the entry in the top-right cell. Define 
    \[
    C_T \coloneqq X^\circ_w \cap Y_{n,k}
    \]
    where $w = [n-1,n-2,\dots,n-i+1,n+s-1,n-i,n-i-1,\dots, 1]$.
\end{definition}


\begin{theorem}[\cite{GLW}]\label{thm:YDim}
    The subspace $C_T$ is a cell of $Y_{n,(1^{n-1}),s}$ in an affine paving which has maximal dimension $\dim(Y_{n,(1^{n-1}),s}) = \binom{n-1}{2}+(s-1)$. Thus, $K^T = \overline{C_T}$ is an irreducible component of $Y_{n,k}$. In fact, all irreducible components are of this form.
\end{theorem}

\begin{example}\label{ex:IrredCmpts}
For example, the following are all standard Young tableau of shape $(2,1,1)/\emptyset=(2,1,1)$ that index the irreducible components of $Y_{4,3}=Y_{4,(1^3),3}$.
\begin{center}
$\ytableaushort{1 4,2,3}$ \;
$\ytableaushort{1 3,2,4}$ \;
$\ytableaushort{1 2,3,4}.$ \;
\end{center}
\end{example}
\begin{example}\label{ex:IrredCmpt2}
The following are all standard Young tableau of shape $(2,1,1)/(1)$ that index the irreducible components of $Y_{3,(1^2),3}$.
\begin{center}
\ytableausetup{notabloids}
$\begin{ytableau} \none & 3\\ 2 \\ 1\end{ytableau}$\;
$\begin{ytableau} \none & 2\\ 3 \\ 1\end{ytableau}$\;
$\begin{ytableau} \none & 1\\ 3 \\ 2\end{ytableau}$.
\end{center}
\end{example}

Note that when $s=n-1$, as in Example \ref{ex:IrredCmpts}, the tableaux are straight shape $(2,1,\dots, 1)$ and there are $n-1$ of them. When $s>n-1$ as in Example \ref{ex:IrredCmpt2}, the tableaux are skew shapes consisting of a single column $(1^{n-1})$ plus a single box, and there are $n$ of them.

\begin{figure}
    \centering
    \ytableausetup{mathmode, boxframe=normal, boxsize=3em}
    \begin{ytableau}
        \none & e_{n-1+s}\\
        \none & \vdots\\
        \none & e_{n+s}\\
        e_{n-1} & e_{2n-2} \\ 
        \vdots & \vdots \\ 
        e_2 & e_{n+1} \\  
        e_{1} &  e_{n}
    \end{ytableau}
    \caption{An illustration of the nilpotent $x$ which sends the vector $e_i$ to the vector in the cell to its left if it exists, or otherwise to $0$.}
    \label{fig:NilpotentPic}
\end{figure}

\begin{remark}\label{rmk:PermutationAlgo}
    The motivation for the definition of the permutation $w$ in Definition~\ref{def:Cells} is as follows: The nilpotent operator $x$ can alternatively be described by the diagram in Figure~\ref{fig:NilpotentPic} 
    in which each vector $e_i$ is mapped under $x$ to its vector immediately to its left, or if $e_i$ has no cell to its left then it's mapped to $0$.

    Given a standard Young tableau $T$ of shape $(2,1^{n-2})/(1^{s-n+1})$, the partial permutation $w$ is defined such that $e_{w_j}$ is the vector corresponding to the cell labeled by $j$ in $T$. In Example~\ref{ex:IrredCmpts}, the three irreducible components correspond, respectively, to the partial permutations $3216$, $3261$, and $3621$.

    Note that our tableau are vertically flipped from the ones in \cite{GLW}, since we have flipped the diagram describing the nilpotent matrix $x$ in order for the irreducible components to be described by tableaux increasing down each column.
\end{remark}

\subsection{Iterated fiber bundles}
An \textbf{iterated fiber bundle} is a sequence of maps of topological spaces $E_m \to E_{m-1} \to \cdots \to E_1$ where each map is a fiber bundle. We will say that the iterated fiber bundle is \textbf{of type} $(E_1,F_1,\dots, F_{m-1})$ if $E_{i+1}\to E_i$ has fiber $F_i$. We will say the iterated fiber bundle is an \textbf{iterated Grassmannian bundle} if for all $i$, $F_i\cong \Gr(k_i,n_i)$ for some positive integers $k_i\leq n_i$ and all maps are algebraic. 

\begin{example}\label{ex:FlagBundle}
For example, $\Fl(1^n)$ is an iterated Grassmannian bundle of type $(\bP^{n-1},\bP^{n-2},\dots,\bP^{1})$, which is realized by the sequence of forgetting maps
\[
\Fl(1^n) = \Fl(1^{n-1},1) \to \Fl(1^{n-2},2)\to \Fl(1^{n-3},3)\to\cdots \to \Fl(1,n-1),
\]
in which $\Fl(1^i,n-i)\to \Fl(1^{i-1},n-i+1)$ forgets the $i^{th}$ subspace and is a Grassmannian bundle with fiber $\Gr(1,n-i+1)\cong \bP^{n-i}$.
\end{example}

\subsection{Poincar\'e polynomials and Borel-Moore homology}

For a complex variety $X$, let $H^i(X)$ denote the $i^{th}$ integral singular cohomology group of $X$, and let $H_i^{BM}(X)$ denote the $i^{th}$ integral Borel-Moore homology group. See~\cite{FultonYoungTab} for more details not provided here. For a smooth variety $X$ of complex dimension $n$, we have Poincar\'e duality
\[
H^i(X) \cong H_{2n-i}^{BM}(X).
\]
If $E$ is a rank $d$ vector bundle over $X$, then 
\[
H_i^{BM}(X) \cong H^{2n-i}(X) \cong H^{2n-i}(E) \cong H_{2d+i}^{BM}(E).
\]

A useful way to compute Borel-Moore homology of a complex variety $X$ is to filter it by subvarieties whose Borel-Moore homology is concentrated in even degrees. Indeed, given an open subspace $U$ of $X$ and its complementary closed subspace $C = X\setminus U$, then we have a long exact sequence
\[
\cdots \to H_i^{BM}(C) \to H_i^{BM}(X) \to H_i^{BM}(U) \to H_{i+1}^{BM}(C)\to\cdots 
\]
It then follows by induction that if there exists a filtration of closed subvarieties $\emptyset = X_0\subseteq X_1\subseteq X_2\subseteq \cdots \subseteq X_m = X$ such that $X_i\setminus X_{i-1}$ has Borel-Moore homology concentrated in even degrees, then 
\[
H_*^{BM}(X) = \bigoplus_i H_*^{BM}(X_i\setminus X_{i-1}).
\]
and all odd degree groups vanish. An important class of filtration is an \textbf{affine paving} in which each $X_i\setminus X_{i-1}$ is isomorphic to $\bC^{k_i}$ for some nonnegative integer $k_i$.

Suppose $X$ is a complex variety whose Borel-Moore homology groups are free $\bZ$-modules. We define $\Poin(X;q)$ to be the generating function of the ranks of the Borel-Moore homology groups,
\[
\Poin(X;q) \coloneqq \sum_i \mathrm{rk}(H_i^{BM}(X)) q^i.
\]
If the complex variety $X$ is both compact and smooth, or if it is compact and has an affine paving, then $\Poin(X;q)$ coincides with the usual Poincar\'e polynomial, which is defined as the generating function of the ranks of the singular cohomology groups (by Poincar\'e duality in the first case, or by the Universal Coefficient Theorem in the second case).

Phrased in this language, if $E$ is a vector bundle of rank $d$ over $X$ as above, then 
\begin{equation}\label{eq:BundlePoincare}
\Poin(E;q) = q^{2d} \Poin(X;q).
\end{equation}
Furthermore, if $X$ is an iterated Grassmannian bundle, $X = E_m\to E_{m-1}\to\cdots \to E_1$ of type $(E_1,F_1,\dots, F_{m-1})$, where $F_i\cong \Gr(k_i,n_i)$, then 
\begin{equation}\label{eq:IteratedPoincare}
\Poin(X;q) = \Poin(E_1;q)\Poin(F_1;q)\cdots \Poin(F_{m-1};q).
\end{equation}
If $X$ has a filtration as above, then
\begin{equation}\label{eq:FiltrationPoincare}
\Poin(X;q) = \sum_i \Poin(X_i\setminus X_{i-1};q)
\end{equation}
and all nonzero terms are of even degree. For this reason, we will often state our formulas in terms of $\sqrt{q}$ instead of $q$.

\begin{example}
    In our example of $X = \bP^n$, we have 
    \[
    \Poin(\bP^n;\sqrt{q}) = 1 + q + \cdots + q^{n-1} \eqqcolon [n]_q
    \]
    and in the case of the flag variety, since it is an iterated projective bundle by Example~\ref{ex:FlagBundle},
    \begin{equation}\label{eq:FlagPoincare}
    \Poin(\Fl(1^n);\sqrt{q}) = \prod_{i=1}^{n-1} \Poin(\bP^i;q) = [1]_q[2]_q\cdots [n]_q \eqqcolon [n]_q!
    \end{equation}
    Furthermore, it's well known that
    \begin{equation}\label{eq:GrassPoincare}
    \Poin(\Gr(k,n);\sqrt{q}) = \frac{[n]_q!}{[k]_q![n-k]_q!} \eqqcolon \multibinom{n}{k}_q,
    \end{equation}
    which follows from the Schubert cell decomposition of the Grassmannian.
\end{example}

\subsection{Chern and Segre classes}

To each complex vector bundle $E$ on a quasiprojective variety $X$ are associated \textbf{Chern classes} $c_i(E)\in H^{2i}(X)$. Let $c(E) = c_0(E) + c_1(E) + \cdots$ be the \textbf{total Chern class} where by definition $c_0(E) \coloneqq 1$. The total Chern class enjoys the following properties: 
\begin{itemize}
    \item Given a map $f:X\to Y$ of quasiprojective varieties and a complex vector bundle $E$ on $Y$, we have $f^*(c(E)) = c(f^*(E))$, where the first $f^*$ is the induced map on cohomology and $f^*(E)$ is the pulled back vector bundle on $X$.
    \item Given a short exact sequence $0\to E'\to E\to E''\to 0$ of vector bundles on $X$, then $c(E)=c(E')c(E'')$.
    \item If $E$ is a complex vector bundle of rank $r$, then $c_i(E) = 0$ for $i>r$.
    \item If $E\cong \mathbb{C}^d$, the trivial bundle of rank $r$, then $c(E) = 1$.
\end{itemize}

To each vector bundle $E$ (or more generally cone) on $X$, we can also associate the \textbf{Segre classes}. These are the unique cohomology classes $s_i(E)\in H^{2i}(X)$ such that if we define the \textbf{total Segre class} $s(E) = s_0(E) + s_1(E) + \cdots$, then
\[
s(E) = c(E)^{-1}.
\]

Given variables $x_1,\dots, x_n$, let $e_d(x_1,\dots, x_n)$ and $h_d(x_1,\dots, x_n)$ be the \textbf{ith elementary symmetric polynomial} and \textbf{ith complete homogeneous symmetric polynomial}, respectively. They are defined as follows,
\begin{align*}
e_d(x_1,\dots, x_n) &= \sum_{1\leq i_1<i_2<\cdots < i_d\leq n} x_{i_1}x_{i_2}\cdots x_{i_d}\\
h_d(x_1,\dots, x_n) &= \sum_{1\leq i_1\leq i_2\leq \cdots \leq i_d\leq n} x_{i_1}x_{i_2}\cdots x_{i_d}.
\end{align*}

The basic facts about these polynomials that we need are the following: If $x_1,\dots, x_r$ are the Chern roots of the vector bundle $E$ of rank $r$, then the total Chern class is
\[
c(E) = (1+x_1)(1+x_2)\cdots (1+x_r) = \sum_{d\geq 0} e_d(x_1,\dots, x_r)
\]
and the total Segre class is
\[
s(E) = \frac{1}{(1+x_1)\cdots (1+x_r)} = \sum_{d\geq 0} (-1)^d h_d(x_1,\dots, x_r).
\]
See \cite{FultonIntTh} for more details on Chern and Segre classes.

\section{Characterization of irreducible components}\label{sec:Irred Cmpts}

In this section we prove our main result, Theorem~\ref{thm:GrassBundle}, that all irreducible components of the $\Delta$-Springer fiber $Y_{n,n-1}$ are iterated Grassmannian fiber bundles. We then prove that all intersections of irreducible components are also iterated Grassmannian fiber bundles.

\subsection{Components as iterated Grassmannian fiber bundles}

Let $T$ be a tableau indexing an irreducible component of $Y_{n,(1^{n-1}),s}$ such that $i$ is the unique label in the second column. We will denote by $K^i\coloneqq K^T$ the corresponding irreducible component.

\begin{theorem}\label{thm:ComponentDescrip}
    The $i^{th}$ irreducible component $K^i$ has the following precise description
    \begin{equation}\label{eq:KiDescrip}
    K^i =\{V_\bullet\in \Fl(1^n,s-1) \mid V_{i-1}\subseteq \im(x) \subseteq V_n \subseteq x^{-1}V_{i-1}\}
    \end{equation}
    Here, $x^{-1}V_{i-1}$ is the preimage of the subspace $V_{i-1}$ under the linear map $x: \bC^{n+s-1}\to \bC^{n+s-1}$.
\end{theorem}

In order to prove Theorem~\ref{thm:ComponentDescrip}, we need to first prove some basic properties of the right-hand side of \eqref{eq:KiDescrip}. Let $Z^i$ denote this subspace of $Y_{n,(1^{n-1}),s}$,
\[Z^i\coloneqq\{V_\bullet\in \Fl(1^n,s-1) \mid V_{i-1}\subseteq \im(x) \subseteq V_n \subseteq x^{-1}V_{i-1} \}.\]
\begin{remark}
    Observe that in the case $s>n-1$, the subspaces $Z^i$ for $i=1,\dots, n$ are all nonempty (in fact we will characterize the permutation flags contained in it in Subsection~\ref{subsec:PermFlags}). 
    
    On the other hand, in the case $s=n-1$, then $Z^1$ in particular is empty. Indeed, if $V_\bullet\in Z^1$, then $V_n\subseteq x^{-1}V_0 = \ker(x)$. However, $V_n$ is $n$-dimensional and $\ker(x)$ is $(n-1)$-dimensional when $s=n-1$, a contradiction. The corresponding ``irreducible component'' $K^1$ is undefined when $s=n-1$ because there is no such standard Young tableau $T$ with $1$ in the top-right cell. Thus, throughout the article we will implicitly assume $2\leq i\leq n$ in the case when $s=n-1$, and otherwise $1\leq i \leq n$.
\end{remark}

\begin{lemma}\label{lem:IteratedGrass}
    The variety $Z^i$ is an iterated Grassmannian fiber bundle of dimension $\binom{n-1}{2} + (s-1)$.
\end{lemma}

\begin{proof}
    A flag $V_\bullet\in Z^i$ is uniquely determined by the following process:
    \begin{enumerate}
        \item Choose $V_{i-1}$ an $(i-1)$-dimensional subspace of $\im(x)$. 
        \item Choose a complete flag in $\Fl(V_{i-1})$. 
        \item Choose an $n$-dimensional subspace $V_n$ such that $\im(x) \subseteq V_n\subseteq x^{-1}V_{i-1}$. Equivalently, choose an element of $\mathbb{P}(x^{-1}V_{i-1}/\im(x))$. 
        \item Choose the remainder of the flag so that $V_{i-1}\subseteq V_i \subseteq V_{i+1}\subseteq\cdots \subseteq  V_n$, or equivalently an element of $\Fl(V_n/V_{i-1})$.
    \end{enumerate}
    In step (3), observe that $\dim(x^{-1}V_{i-1}) = (i-1)+s$ for all choices of $V_{i-1}$. Thus, once $V_{i-1}$ is chosen then $\dim(x^{-1}V_{i-1}/\im(x)) = (i-1)+s-(n-1) = s+i-n$. Furthermore, $V_n/\im(x)$ has dimension $n-(n-1)=1$, so the choice of $V_n$ in step (3) is equivalent to a choice of element of $\bP(x^{-1}V_{i-1}/\im(x))$, as claimed. 
    
    Therefore, $Z^i$ is an iterated Grassmannian fiber bundle of type $\Gr(i-1,n-1)$, $\Fl(1^{i-1})$, $\mathbb{P}^{s+i-n-1}$, and $\Fl(1^{n-i+1})$, respectively corresponding to each step above. 
    The dimension of $Z^i$ is the sum of the dimensions of the fibers, which is
    \begin{align*}
    \dim(Z^i) & =  (i-1)(n-i) + \binom{i-1}{2} + (s+i-n-1) + \binom{n-i+1}{2}\\
        &= \frac{n(n-3)}{2}+s\\
        &= \binom{n-1}{2} + (s-1),
    \end{align*}
    and the proof is complete.
\end{proof}

\begin{lemma}\label{lem:Smoothsubspace}
    The iterated Grassmannian bundle $Z^i$ is a smooth closed subspace of $Y_{n,(1^{n-1}),s}$ of dimension $\binom{n-1}{2} + (s-1)$.
\end{lemma}

\begin{proof}
    The smoothness property and the dimension formula follow immediately by Lemma~\ref{lem:IteratedGrass}. It suffices to show that $Z^i$ is a subspace of $Y_{n,1^{n-1},s}$ and is closed. 

    Let $V_{\bullet} \in Z^i$. Observe that $\im(x)\subseteq V_n$ is immediate from the definition of $Z^i$. It suffices to show $xV_j\subseteq V_j$ for all $j\leq n$. For $j\leq i-1$, $V_j \subseteq V_{i-1} \subseteq \im(x)\subseteq \ker(x)$, so $xV_j = 0\subseteq V_j$. For $i\leq j\leq n$, observe that $V_n\subseteq x^{-1}V_{i-1}$ is equivalent to $xV_n\subseteq V_{i-1}$, and hence
    \[
    xV_j\subseteq xV_n \subseteq V_{i-1}\subseteq V_j.
    \]
    Thus, $Z^i\subseteq Y_{n,(1^{n-1}),s}$. To see that $Z^i$ is closed in $Y_{n,(1^{n-1}),s}$, observe that $V_{i-1}\subseteq \im(x)$ is a closed condition, and $V_n\subseteq x^{-1}V_{i-1}$ is equivalent to $xV_n \subseteq V_{i-1}$ which is also a closed condition.
\end{proof}

\begin{proposition}\label{prop:ComponentContain}
    We have the inclusion $K^i\subseteq Z^i$.
\end{proposition}
\begin{proof}
    Recall that $K^i = K^T=\overline{C_T}$ where $C_T = X_w^\circ \cap Y_{n,(1^{n-1}),s}$ such that $w=[n-1,n-2,\dots,n-i+1,n+s-1,n-i,\dots, 1]$ and $i$ is the entry in the top-right box of the standard Young tableau $T$ of shape $(2,1^{n-2})/(1^{s-n+1})$.
    In Lemma~\ref{lem:Smoothsubspace}, we found $Z^i$ to be closed, so it suffices to show $C_T\subseteq Z^i$.

        The matrix representatives of the flags in $C_T$ are of the following form:
        \[
        \begin{bNiceMatrix}[t][first-row,first-col]
            &              &           &         & f_i       &        &      &  \\
        1 & *         & \dots     & *        & *       & *       & \dots   & 1  \\
  \vdots & \vdots    &           & \vdots   & \vdots  &\vdots  & \iddots &    \\
    n-i & \vdots    &           & *        & *       & 1      &         &    \\
   n-i+1 & \vdots    & \iddots   & 1        & 0       &        &         &    \\
         & *         & \iddots   &          &         &        &         &    \\
       n-1 & 1         &           &          &         &        &         &    \\
        \hline
     n &           &           &          & *       &        &         &    \\
         &           &           &          & \vdots  &        &         &    \\
       n+s-1 &           &           &          & 1       &        &         &    \\
         \CodeAfter
            \OverBrace[shorten,yshift=3pt]{1-1}{2-3}{i-1}
            \OverBrace[shorten,yshift=3pt]{1-5}{5-7}{n-i}
        \end{bNiceMatrix}
        \]
    where the entries $\ast$ are free parameters subject to the conditions defining $Y_{n,(1^{n-1}),s}$, and the columns are $f_j$. By the definition of $Y_{n,(1^{n-1}),s}$, we must have $xf_i\in V_{i-1}$, and it is not hard to check that this is the only condition on the columns.


    We also check that the defining property of $Z^i$ is satisfied for $C_T$. 
    First, $V_{i-1}\subseteq \im(x)$ since by the definition of the Schubert cell $X_w^\circ$, the columns $f_1,\dots, f_{i-1}$ are in $\vspan\{e_1,\dots, e_{n-1}\}=\im(x)$. 
    Second, $\im(x) \subseteq V_n$ since all rows $1,\dots,n-1$ contain a pivot. So,  $\im(x) = \vspan\{e_1,\dots,e_{n-1}\} \subseteq V_n$. 
    Finally, we can see that $V_n\subseteq x^{-1}V_{i-1}$ by inspecting the equivalent containment $xV_n\subseteq V_{i-1}$, which in turn is equivalent to the inclusion $\vspan(xf_i)\subseteq V_{i-1}$, which is satisfied by $C_T$. Hence, we can conclude $C_T\subseteq Z^i$, and thus $K^i\subseteq Z^i$.
\end{proof}

\begin{proof}[Proof of Theorem~\ref{thm:ComponentDescrip}]
    The space $Z^i$ is irreducible because it is an iterated Grassmannian bundle. By Proposition~\ref{prop:ComponentContain}, $K^i \subseteq Z^i$.
    Therefore, it remains to show that $\dim K^i = \dim Z^i$. The dimension of $K^i$ is simply the dimension of the variety $Y_{n,(1^{n-1}),s}$, which is $(s-1)+{n-1\choose2}$ by Theorem \ref{thm:YDim}. We showed in Lemma~\ref{lem:IteratedGrass} that the dimension of $Z^i$ has the same formula (when it is nonempty), and thus $K^i=Z^i$.
\end{proof}

\begin{remark}
    Let us observe that by the characterization in Theorem~\ref{thm:ComponentDescrip}, each irreducible component is a nilpotent Hessenberg variety in $\Fl(1^n,s-1)$. Indeed, $K^i$ may be rewritten as
    \[
    K^i = \{V_\bullet\in \Fl(1^n,s-1)\mid xV_{i-1}\subseteq V_0,\, xV_n \subseteq V_{i-1}\}.
    \]
    Thus, $K^i$ is the Hessenberg variety associated to the operator $x$ and Hessenberg function $h = (0,0,\dots, 0,i-1,\dots, i-1,n+s-1)$.
\end{remark}

\begin{theorem}\label{thm:GrassBundle}
    The irreducible component $K^i$ is an iterated Grassmannian bundle of type $\Gr(i-1,n-1)$, $\Fl(1^{i-1})$, $\mathbb{P}^{s+i-n-1}$, $\Fl(1^{n-i+1})$.
\end{theorem}

\begin{proof}
    This follows immediately from Theorem~\ref{thm:ComponentDescrip} and the proof of Lemma~\ref{lem:IteratedGrass}.
\end{proof}

\subsection{Permutation flags}\label{subsec:PermFlags}

By \cite[Lemma 3.11]{GLW}, the permutation flags $V_\bullet$ in $Y_{n,(1^{n-1}),s}$ are in bijection with row-increasing partial labelings of the cells of the shape $(1^{s-n+1},2^{n-1})$ (as drawn in Remark \ref{rmk:PermutationAlgo}) in which $n$ of the cells are labeled $1,2,\dots, n$, there is no empty cell to the left of a labeled cell, and all cells in the first column are labeled. Precisely, let $V_\bullet$ be the permutation flag associated to the permutation $w$, so that $V_i = \vspan\{e_{w_1},\dots, e_{w_i}\}$ for all $i$. Then define a labeling such that the cell corresponding to $e_{w_i}$ in the diagram of Remark~\ref{rmk:PermutationAlgo} is labeled $i$. It is easy to check that this labeling has the desired properties. 

For example, the first two partial row-increasing fillings pictured in the top row of Figure~\ref{fig:Cells1} correspond to the permutation flag indexed by $w=3214$ and $w=3124$, respectively. The first partial row-increasing filling in the second row of Figure~\ref{fig:Cells2} corresponds to $w=241$.

Next, we give a characterization of the permutation flags contained in each irreducible component $K^i$.

\begin{corollary}
    The permutation flags in the irreducible component $K^i$ are in bijection with row-increasing partial fillings such that either (1) there is a unique row containing two letters filled with $a,b$ such that $a < i \leq b$ or (2) the unique letter $b$ in the second column has no cell to its left and $i\leq b$.
\end{corollary}
\begin{proof}
    Let $V_\bullet$ be a permutation flag in $Y_{n,(1^{n-1}),s}$ corresponding to a partial permutation $w$, meaning that $V_i =\vspan\{e_{w_1},\dots,e_{w_i}\}$ for all $i$. For case (1), suppose $w$ corresponds to a row-increasing filling $T$ whose unique row containing two letters is filled with $a<b$. It suffices to show that $V_\bullet\in K^i$ if and only if $a<i\leq b$.

    If $V_\bullet \in K^i$, then by the definition of $K^i$, $V_{i-1} \subseteq \im(x) \subseteq V_n \subseteq x^{-1}V_{i-1}$. If $i > b$, then $b \leq i-1$ and so $e_{w_b}\in \im(x) = \vspan\{e_1,\dots, e_{n-1}\}$, which means that the label $b$ must be in the first column of $T$, a contradiction since it must be to the right of $a$. It follows that $i \leq b$.
    
    The space $V_a$ is spanned by $\{e_{w_1},\dots,e_{w_a}\}$ and $V_b$ is spanned by $\{e_{w_1},\dots,e_{w_b}\}$. Since $b$ is in the same row as $a$ (and to the right of $a$), then $w_b = w_a + (n - 1)$. It follows from the fact that $xV_n \subseteq V_{i-1}$ that $xV_b \subseteq V_{i-1}$. Thus $xe_{w_b} = e_{w_a} \in V_{i-1}$. Since $V_{i-1}$ is spanned by $\{e_{w_1},\dots, e_{w_{i-1}}\}$, then we must have that $a\leq i-1$ so $a < i$. This then shows that $a<i\leq b$.
    
    Now we suppose $a < i \leq b$ and show $V_\bullet \in K^{i}$. Indeed, $i\leq b$ implies $i-1<b$. Since $b$ is the only label in the second column of $T$, then the labels $1,\dots, i-1$ are in the first column of $T$, and thus $V_{i-1} \subseteq \im(x)$. Furthermore, since $V_\bullet$ is a permutation flag, $b$ is the unique label in the second column of $T$, and $xe_{w_b}=e_{w_a}$, we have $xV_n = xV_b\subseteq V_a$ and it follows from $a<i$ that $V_a\subseteq V_{i-1}$, and so $xV_n\subseteq V_{i-1}$. Equivalently, $V_n \subseteq x^{-1}V_{i-1}$, and thus $V_\bullet\in K^i$.

    We conclude that $V_\bullet\in K^i$ if and only if $a < i \leq b$, proving the corollary for case (1).

    Alternatively, for case (2), let $w$ correspond instead to a row-increasing filling such that the unique cell in the second column is filled with $b$, and there is no cell to its left. Then the condition that the first column of $T$ is entirely filled is equivalent to the condition that $\im(x)\subseteq V_n$, and the condition that $i\leq b$ is equivalent to $V_{i-1}\subseteq \im(x)$. Since $b$ has no cell to its left, then $V_n\subseteq x^{-1}V_{i-1}$ is automatic from $xe_{w_b}=0$.
\end{proof}

\begin{figure}
    \centering
    \ytableausetup{mathmode, boxframe=normal, boxsize=normal}
        \begin{ytableau}1&4\\2 & \\ 3 & \end{ytableau}
        \,
        \begin{ytableau}1&4\\3 & \\ 2 & \end{ytableau}
        \,
        \begin{ytableau}2&\\ 1&4 \\ 3 & \end{ytableau}
        \,
        \begin{ytableau}3&\\ 1&4 \\ 2 & \end{ytableau}
        \,
        \begin{ytableau}3 & \\2 & \\ 1&4 \end{ytableau}
        \,
        \begin{ytableau}2 & \\3 & \\ 1&4 \end{ytableau}
        \,

        \vspace{0.5cm}

        \begin{ytableau}2 & 4\\ 1& \\ 3 & \end{ytableau}
        \,
        \begin{ytableau}2 & 4\\ 3& \\ 1 & \end{ytableau}
        \,
        \begin{ytableau}1&\\2 &4 \\ 3 & \end{ytableau}
        \,
        \begin{ytableau}3&\\2 &4 \\ 1 & \end{ytableau}
        \,
        \begin{ytableau}1& \\ 3 & \\2 &4 \end{ytableau}
        \,
        \begin{ytableau}3& \\ 1 & \\2 &4 \end{ytableau}
        \,
        
        \vspace{0.5cm}
        
        \begin{ytableau}2 & 3\\1 & \\ 4 & \end{ytableau}
        \,
        \begin{ytableau}2 & 3\\4 & \\ 1 & \end{ytableau}
        \,
        \begin{ytableau}1&\\2 & 3 \\ 4 & \end{ytableau}
        \,
        \begin{ytableau}4&\\2 & 3 \\ 1 & \end{ytableau}
        \,
        \begin{ytableau}1&\\4& \\ 2 & 3 \end{ytableau}
        \,
        \begin{ytableau}4&\\1& \\ 2 & 3 \end{ytableau}
        \,

        \vspace{0.5cm}

        \begin{ytableau}1&3\\2 & \\ 4 & \end{ytableau}
        \,
        \begin{ytableau}1&3\\4 & \\ 2 & \end{ytableau}
        \,
        \begin{ytableau}2 &\\1&3 \\ 4 & \end{ytableau}
        \,
        \begin{ytableau}4 &\\1&3 \\ 2 & \end{ytableau}
        \,
        \begin{ytableau}2 &\\4 & \\ 1&3 \end{ytableau}
        \,
        \begin{ytableau}4 &\\2 & \\ 1&3 \end{ytableau}
        \,
    \caption{The partial row-increasing fillings corresponding to the permutation flags in the irreducible component $K^{3}$ of $Y_{4,3}=Y_{4,(1^3),3}$.}
    \label{fig:Cells1}
\end{figure}

\begin{figure}
    \centering
        \begin{ytableau} \none & 3\\1 & \\ 2 & \end{ytableau}
        \,
        \begin{ytableau} \none & 3\\2 & \\ 1 & \end{ytableau}
        \,
        \begin{ytableau} \none & 2\\1 & \\ 3 & \end{ytableau}
        \,
        \begin{ytableau} \none & 2\\3 & \\ 1 & \end{ytableau}
        
        \vspace{0.5cm}
        
        \begin{ytableau} \none & \\1 & 2\\ 3 & \end{ytableau}
        \,
        \begin{ytableau} \none & \\1 & 3\\ 2 & \end{ytableau}
        \,
        \begin{ytableau} \none & \\2 & \\ 1 & 3\end{ytableau}
        \,
        \begin{ytableau} \none & \\3 & \\ 1 & 2\end{ytableau}
    \caption{The partial row-increasing fillings corresponding to permutation flags in the irreducible component $K^2$ of $Y_{3,(1^2),3}$.}
    \label{fig:Cells2}
\end{figure}

\subsection{Intersections of irreducible components}
 For $1\leq i<j\leq n$ in the case $s>n-1$, or $2\leq i<j\leq n$ when $s=n-1$, let
\[
K^{i,j}\coloneqq K^i\cap K^{j}.
\]
By Theorem~\ref{thm:ComponentDescrip}, we then have
\begin{equation}\label{eq:IntersectionDescrip}
K^{i,j} = \{V_\bullet\in \Fl(1^n,s-1)\mid V_{j-1}\subseteq \im(x) \subseteq V_n\subseteq x^{-1}V_{i-1}\}.
\end{equation}

\begin{proposition}\label{prop:Intersections are bundles}
    Let $b_1< b_2<\cdots < b_m\leq n$. Then
    \[
        K^{b_1}\cap \cdots \cap K^{b_m} = K^{b_1,b_m},
    \]
    which is an iterated Grassmannian bundle of type $\Gr(b_m-1,n-1)$, $\Fl(1^{b_m-1})$, $\mathbb{P}^{s+b_1-n-1}$, $\Fl(1^{n-b_m+1})$, which is of dimension $\binom{n-1}{2} + s - (b_m-b_1)$.
\end{proposition}

\begin{proof}
    The first claim follows from~\ref{eq:IntersectionDescrip}. The description as an iterated Grassmannian bundle then follows by a similar analysis as in the proof of Lemma~\ref{lem:IteratedGrass}. The dimension follows from the following calculation. The total dimension is the sum of the dimensions of the fibers, 
    \begin{align*}
        &\dim(\Gr(b_m-1,n-1))+\dim(\Fl(1^{b_m-1}))+\dim(\mathbb{P}^{s+b_1-n-1})+ \dim(\Fl(1^{n-b_m+1}))\\
        &=(b_m-1)(n-b_m) + \binom{b_m-1}{2} + (s+b_1-n-1) + \binom{n-b_m+1}{2}\\
        &= \binom{n-1}{2}+s - (b_m - b_1).
    \end{align*}
\end{proof}

\begin{remark}\label{rmk:RootPoset}
    By Proposition~\ref{prop:Intersections are bundles}, we see that the poset formed by the set of intersections of irreducible components of $Y_{n,n-1}=Y_{n,(1^{n-1}),n-1}$ under reverse inclusion is isomorphic to the type $A_{n-1}$ root lattice. Indeed, the type $A_{n-1}$ root lattice is the poset whose elements are the positive roots $\alpha_{i,j} = e_i-e_{i+1}$, with $i<j$, where $e_i$ is the $i^{th}$ coordinate vector in $\mathbb{R}^n$. We have $\alpha_{i,j}\geq \alpha_{k,\ell}$ in the root lattice if and only if  $i\leq k<\ell\leq j$. Thus, we have a poset isomorphism given by associating the component intersection $K^{i,j}$ with the root $\alpha_{i-1,j-1}$. 
    
    Similarly, for $s>n-1$ the set of intersections of irreducible components  of $Y_{n,(1^{n-1}),s}$ under reverse inclusion is isomorphic to the $A_n$ root lattice.
\end{remark}

\begin{example}
For example, the poset of irreducible component intersections of $Y_{4,3}$ is
\[
\begin{tikzcd}
    &&K^{2}\cap K^{4}&&\\
    &K^{2}\cap K^{3}\arrow[dash,ru] && K^{3}\cap K^{4}\arrow[dash,lu]&\\
    K^{2}\arrow[dash,ru] && K^{3}\arrow[dash,lu]\arrow[dash,ru] && K^{4}.\arrow[dash,lu]
\end{tikzcd}
\]
\end{example}

\section{Poincar\'e polynomials and Dyck paths}\label{sec:Poincare}

In this section, we give formulas for the Poincar\'e polynomials of unions of intersections of irreducible components of $Y_{n,n-1}$ in terms of arm and leg statistics on Dyck paths. All of the results in this section can easily be extended to the setting of $Y_{n,(1^{n-1}),s}$ for $s>n-1$; for simplicity of notation we leave this to the reader and state our results only in the case $Y_{n,n-1}$.

We first state a formula for the Poincar\'e polynomial of an irreducible intersection $K^{i,j}$ in $Y_{n,n-1}$. We then give a formula for the Poincar\'e polynomial of an arbitrary union of $K^{i,j}$. 

\begin{proposition}\label{prop:Poincare-ij}
For $2\leq i<j\leq n$, 
\[
\Poin(K^{i,j};\sqrt{q}) = [n-1]_q! [i-1]_q[n-j+1]_q.
\]
\end{proposition}

\begin{proof}
Given $2\leq i<j\leq n$, by Proposition~\ref{prop:Intersections are bundles},  $K^{i,j}$ is an iterated Grassmannian bundle of type $\Gr(j-1,n-1)$, $\Gr(1,i-1)$, $\Fl(1^{j-1})$, $\Fl(1^{n-j+1})$. Therefore, by \eqref{eq:IteratedPoincare},
\begin{align*}
    \Poin(K^{i,j}, \sqrt{q}) &= \multibinom{n-1}{j-1}_{q} \cdot [i-1]_{q} \cdot [j-1]_{q}! \cdot [n-j+1]_{q}! \\
    &= [n-1]_{q}! \cdot [i-1]_{q} \cdot [n-j+1]_{q}.
\end{align*}
\end{proof}

Given a cell in the $n\times n$ grid, we will label the cell in the $i^{th}$ column from the left and $j^{th}$ row from the bottom by $(i+1,j)$ as in the left side of Figure~\ref{fig:DyckPath} 

\begin{remark}
    The reasoning for this unconventional labeling with $i+1$ instead of $i$ is to match the indexing of the $K^{i,j}$ starting with $i= 2$. When working with $Y_{n,(1^{n-1}),s}$ for $s>n-1$, this must be replaced with the $(n+1)\times (n+1)$ grid in which the cell in the $i^{th}$ column from the left and the $j^{th}$ row from the bottom is labeled by $(i,j-1)$ for $j>1$.
\end{remark}

Let $D$ denote a \textbf{Dyck path} in the $n\times n$ grid, which is a lattice path starting from the lower left corner taking only North and East unit steps that stays weakly above the diagonal. We may alternatively identify $D$ with the set of cells in the grid that lie \textit{above} the Dyck path $D$. Given a cell $c \in D$ lying above the path with coordinates $(i,j)$, its \textbf{arm} $a(c)=i-2$ is the number of cells strictly to the left of $c$ in the grid, and its \textbf{leg} $\ell(c)=n-j$ is the number of cells strictly above $c$ in the grid.

\begin{theorem}\label{thm:Poincare}
    Let $(i_1,j_1),\dots, (i_m,j_m)$ be a sequence of pairs $2\leq i_t<j_t\leq n$. We have
    \begin{equation}\label{eq:DyckPoincare}
    \Poin\left(\bigcup_{t=1}^m K^{i_t,j_t};\sqrt{q}\right) = 
    [n-1]_q!\sum_{c\in D} q^{a(c)+\ell(c)},
    \end{equation}
    where $D$ is the Dyck path formed by the union of cells weakly above and to the left of some $(i_r,j_r)$.
\end{theorem}

\begin{example}
    For $n=5$, $K^{23}\cup K^{44}$ corresponds to the Dyck path drawn in blue in the $5\times 5$ grid in Figure~\ref{fig:DyckPath}. On the left, we have labeled cells by their coordinates and on the right we have labeled them by their $q$ power contribution to \eqref{eq:DyckPoincare}.
    So, the Poincar\'e polynomial has the formula
    \begin{align*}
        \Poin(K^{23} \cup K^{44};\sqrt{q}) &= [4]_q!\,(1+2q+3q^2+q^3)\\
        &= q^9 + 6q^8 + 16q^7 + 28q^6 + 36q^5 + 35q^4 + 26q^3 + 14q^2 + 5q + 1.
    \end{align*}
\end{example}

\begin{figure}
    \centering
    \includegraphics[scale=0.6]{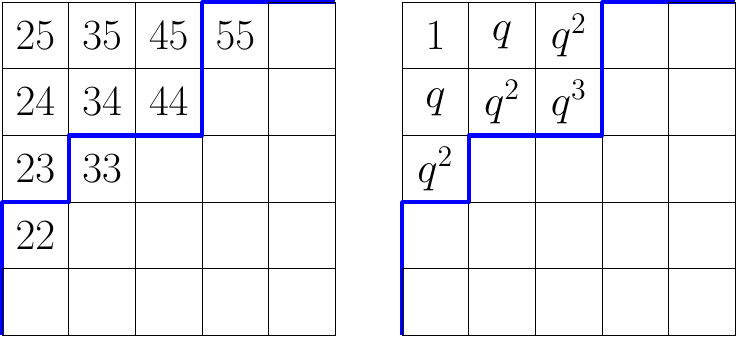}
    \caption{On the left, the Dyck path corresponding to $K^{23}\cup K^{44}$. On the right, the corresponding $q$ power contributions of the cells above the path to the Poincar\'e polynomial.}
    \label{fig:DyckPath}
\end{figure}

Our strategy is to filter the union $\bigcup_{t=1}^m K^{i_t,j_t}$ into subspaces whose successive differences are vector bundles over flag varieties and whose contribution to the Poincar\'e polynomial are the terms in the right-hand side of \eqref{eq:DyckPoincare}. Namely, we have the following Lemma.

\begin{lemma}\label{lem:SetDiffBundle}
    For $2\leq i<j\leq n$, the space $K^{i,j}\setminus \left(K^{i-1,j}\cup K^{i,j+1}\right)$ is a rank $(n-2-(j-i))$ vector bundle over $\Fl(1^{n-1})$. 
    Here, we interpret $K^{1,j}=K^{i,n+1}=\emptyset$.
\end{lemma}

\begin{proof}
    By~\eqref{eq:IntersectionDescrip}, 
    \[
    K^{i,j}\setminus \left(K^{i-1,j}\cup K^{i,j+1}\right) = \{V_\bullet\mid V_{j-1}\subseteq \im(x) \subseteq V_n\subseteq x^{-1}V_{i-1},\, V_j\not\subseteq \im(x),\, V_n\not\subseteq x^{-1}V_{i-2}\}.
    \]
    Letting $v_1,\dots, v_n$ be representative spanning vectors for $V_\bullet$, then $v_j\notin \im(x)$, and $V_n = \im(x) + \mathrm{span}\{v_j\}$. Thus, $\dim(V_\ell \cap \im(x)) = \ell-1$ for $\ell > j$. 

    Define a map
    \[
    \phi: K^{i,j}\setminus \left(K^{i-1,j}\cup K^{i,j+1}\right) \to \Fl(\im(x))
    \]
    by $\phi(V_\bullet) = (V_1,\dots, V_{j-1},V_{j+1}\cap \im(x),\dots, V_n\cap \im(x))$. Given $V_\bullet' \in \Fl(\im(x))$, then the fiber over $V_\bullet'$ is the set of all $V_\bullet$ in the domain such that 
    \begin{gather*}
    V_1=V_1',\dots, V_{j-1}=V_{j-1}',\\
    V_{j+1}\cap \im(x) = V_j',\dots, V_n\cap \im(x) = V_{n-1}'.
    \end{gather*}
    Then the fiber over $V_\bullet'$ is determined by the choice of $v_j$. Since $V_n\subseteq x^{-1}V_{i-1}\setminus x^{-1}V_{i-2}$, then the choice of $v_j$ is equivalent to a choice of nonzero vector in
    \[
    (x^{-1}V_{i-1}/V_{j-1})\setminus (x^{-1}V_{i-2}/V_{j-1}).
    \] 
    Furthermore, $\dim(x^{-1}V_{i-1}/V_{j-1}) = (n-1)+(i-1)-(j-1) = n-1-(j-i)$. Thus, we have an isomorphism
    \[
    \phi^{-1}(V_\bullet) \cong \mathbb{P}(x^{-1}V_{i-1}/V_{j-1})\setminus \mathbb{P}(x^{-1}V_{i-2}/V_{j-1})\cong \bC^{n-2-(j-i)},
    \]
    which is a vector space that varies continuously over $\Fl(\im(x))$. Hence, $\phi$ is a rank $n-2-(j-i)$ vector bundle.
\end{proof}

\begin{corollary}\label{cor:SetDiffPoincare}
    For $2\leq i<j\leq n$, letting $c=(i,j)$, then
    \[
    \Poin(K^{i,j}\setminus (K^{i-1,j}\cup K^{i,j+1});\sqrt{q}) = [n-1]_q!\,q^{a(c)+\ell(c)}.
    \]
\end{corollary}

\begin{proof}
    Since $K^{i,j}\setminus (K^{i-1,j}\cup K^{i,j+1})$ is a rank $(n-2-(j-i))$ vector bundle over $\Fl(1^{n-1})$ by Lemma~\ref{lem:SetDiffBundle}, then by \eqref{eq:BundlePoincare} and \eqref{eq:FlagPoincare}, we have
    \begin{equation}\label{eq:PoincareDifferencePiece}
    \Poin(K^{i,j}\setminus (K^{i-1,j}\cup K^{i,j+1});\sqrt{q}) = [n-1]_q!\, q^{n-2-(j-i)} = [n-1]_q!\, q^{a(c)+\ell(c)}.
    \end{equation}
\end{proof}

 \begin{proof}[Proof of Theorem~\ref{thm:Poincare}]
Recall that by Remark~\ref{rmk:RootPoset}, the collection of intersections of irreducible components forms a poset under reverse inclusion that is isomorphic to the type $A_{n-1}$ root lattice. Let $P$ be the subposet restricted to the union of all elements weakly above the $K^{i_t,j_t}$ for $t=1,\dots, m$, and take a linear extension of this poset $K^{c_1},\dots, K^{c_{m'}}$ (so that $K^{c_1}$ is a minimal element $K^{i_t,j_t}$ for some $t$, and $K^{c_{m'}}=K^{2,n}$ is the unique maximal element). Observe that the $c_t$ are exactly the cells lying weakly above and to the left of some $(i_r,j_r)$, which are the cells lying above the Dyck path $D$.

Then the chain of unions
\[
K^{c_{m'}} \subseteq (K^{c_{m'}-1}\cup K^{c_{m'}}) \subseteq \cdots \subseteq (K^{c_1}\cup \cdots \cup K^{c_{m'}-1}\cup K^{c_{m'}})
\]
is a filtration of 
\[
\bigcup_{t=1}^{m'} K^{c_t} = \bigcup_{t=1}^m K^{i_t,j_t}.
\]
by closed subvarieties. Suppose $c_t$ has coordinates $c_t = (x_t,y_t)$. Then $K^{c_t} = K^{x_t,y_t}$ is covered by $K^{x_t-1,y_t}$ and $K^{x_t,y_t+1}$ in the poset (which both must appear in the linear extension), and the successive differences of the pieces of the filtration are exactly the subspaces $K^{x_t,y_t}\setminus (K^{x_t-1,y_t}\cup K^{x_t,y_t+1})$ running over all $t$.

Thus, by \eqref{eq:FiltrationPoincare}, we may compute the Poincar\'e polynomial of the union as the sum
\begin{align*}
\Poin\left(\bigcup_{t=1}^m K^{i_t,j_t};\sqrt{q}\right) 
&= \sum_{t=1,\dots,m'} \Poin(K^{x_t,y_t}\setminus (K^{x_t-1,y_t}\cup K^{x_t,y_t+1});\sqrt{q}) \\
&= \sum_{t=1,\dots, m'} [n-1]_q! \, q^{n-2-(y_t-x_t)} \\
&= \sum_{c\in D} [n-1]_q! \, q^{a(c)+l(c)} 
\end{align*}
from which Theorem~\ref{thm:Poincare} follows. Here, we have applied Corollary \ref{cor:SetDiffPoincare} in the second line.
\end{proof}

\section{Singular cohomology of the irreducible components}\label{sec:SingularCohomology}

In this section, we specialize to the case $s=n-1$ and give a presentation of the integral singular cohomology ring of each irreducible component $K^i$ of $Y_{n,n-1}$. Recall the definition \eqref{eq:YnkDef} of $Y_{n,n-1}$ and that in this case, $\im(x)=\ker(x)$.

\begin{theorem}\label{thm:QuotientRing}
    We have
    \[
    H^*(K^i) \cong \bZ[x_1,\dots, x_n]/I_{n}^i
    \]
    where $I_n^i$ is the ideal generated by 
    \begin{itemize}
        \item $e_2(x_1,\dots, x_n),\dots, e_n(x_1,\dots, x_n)$
        \item $h_{j}(x_1,\dots, x_{i-1})$ for $j \geq n+1-i$
        \item $h_j(x_i,\dots, x_n)$ for $j\geq i-1$.
    \end{itemize}
\end{theorem}

\begin{lemma}\label{lem:CohMap}
    There is a surjective map 
    \[
    \mathbb{Z}[x_1,\dots, x_n]/I_n^i \twoheadrightarrow H^*(K^i).
    \]
\end{lemma}

\begin{proof}
    First, we have a map
    \[
    \varphi:\mathbb{Z}[x_1,\dots, x_n] \to H^*(K^i)
    \]
    given by sending $x_j$ to the first Chern class of the tautological quotient line bundle, $c_1(\widetilde{V}_j/\widetilde{V}_{j-1})$ where $\widetilde{V}_j$ is the rank $j$ tautological bundle on $K^i$ whose fiber over $V_\bullet$ is $V_j$. By a slight abuse of notation, we will write $x_j = c_1(\widetilde{V}_j/\widetilde{V}_{j-1})$ in our calculations below. By the construction of $K^i$ as an iterated Grassmannian bundle in Theorem~\ref{thm:GrassBundle}, then $H^*(K^i)$ is generated by the Chern roots $x_1,\dots, x_n$ of $\widetilde{V}_n$. Thus, the map $\varphi$ is surjective.

    Second, we claim that $I_n^i \subseteq \ker(\varphi)$. Indeed, since $\ker(x) \subseteq V_n$. for all $V_\bullet \in Y_{n,n-1}$, then the rank $n-1$ trivial bundle $\ker(x)$ is a sub-bundle of $\widetilde{V}_n$. Therefore, by multiplicativity,
    \[
    c(\widetilde{V}_n) = c(\ker(x)) c(\widetilde{V}_n/\ker(x)) = c(\widetilde{V}_n/\ker(x)).
    \]
    Since $\widetilde{V}_n/\ker(x)$ is a rank $1$ vector bundle, then $c_j(\widetilde{V}_n)=0$ for $j\geq 2$. Since $c_j(\widetilde{V}_n) = e_j(x_1,\dots,x_n)$, then $e_j(x_1,\dots, x_n) \in \ker(\varphi)$ for $j\geq 2$. 
    
    By the definition of $K^i$, we have $V_{i-1}\subseteq \ker(x)$ for all $V_\bullet\in K^i$. Then $\widetilde{V}_{i-1}$ is a sub-bundle of the trivial bundle $\ker(x)$, so we have
    \[
    1=c(\ker(x)) = c(\widetilde{V}_{i-1}) c(\ker(x)/\widetilde{V}_{i-1}).
    \]
    Therefore,
    \[
    s(\widetilde{V}_{i-1}) = c(\ker(x)/\widetilde{V}_{i-1})
    \]
    but $\ker(x)/\widetilde{V}_{i-1}$ has rank $n-i$, so 
    \[
    h_j(x_1,\dots, x_i) = s_j(\widetilde{V}_{i-1}) = c_j(\ker(x)/\widetilde{V}_{i-1}) = 0.
    \]
    Thus, $h_j(x_1,\dots, x_i)\in \ker(\varphi)$ for $j > n-i$. \\

    Third, for each $V_\bullet\in K^i$, since $V_{i-1}\subseteq \ker(x)$, then  $x^{-1}V_{i-1}$ has dimension $(n-1) + (i-1)$. Let $\widetilde{x^{-1}V_{i-1}}$ be the corresponding tautological vector bundle of rank $(n-1) + (i-1)$. Observe that by the definition of $K^i$, we have a short exact sequence
    \[
    0 \to \widetilde{V}_n/\widetilde{V}_{i-1} \to \widetilde{x^{-1}V}_{i-1}/\widetilde{V}_{i-1} \to  \widetilde{x^{-1}V}_{i-1}/\widetilde{V}_n \to 0.
    \]
    Thus, 
    \begin{equation}\label{eq:QuotientChern}
    c(\widetilde{x^{-1}V}_{i-1}/\widetilde{V}_n) = \frac{c(\widetilde{x^{-1}V}_{i-1}/\widetilde{V}_{i-1})}{c(\widetilde{V}_n/\widetilde{V}_{i-1})}
    \end{equation}
    Furthermore, we have another short exact sequence
    \[
    0 \to \ker(x) \to \widetilde{x^{-1} V}_{i-1} \xrightarrow[]{x} \widetilde{V}_{i-1}\to 0
    \]
    where the second map is induced by $x$. Thus,
    \[
    c(\widetilde{x^{-1} V}_{i-1}/\widetilde{V}_{i-1}) = c(\widetilde{x^{-1} V}_{i-1})/c(\widetilde{V}_{i-1}) = c(\ker(x)) = 1
    \]
    Combining this with \eqref{eq:QuotientChern}, since the Chern roots of $\widetilde{V}_n/\widetilde{V}_{i-1}$ are $x_i,\dots, x_n$ we have
    \[
    c(\widetilde{x^{-1}V}_{i-1}/\widetilde{V}_n) = \frac{1}{c(\widetilde{V}_n/\widetilde{V}_{i-1})} = s(\widetilde{V}_n/\widetilde{V}_{i-1}) = \sum_j (-1)^j h_j(x_i,\dots, x_n).
    \]
    Since $\widetilde{x^{-1}V}_{i-1}/\widetilde{V}_n$ has rank $((n-1)+(i-1))-n = i-2$, then $h_j(x_i,\dots, x_n) = 0$ for all $j\geq i-1$. Therefore, $h_j(x_i,\dots, x_n)\in I_n^i$ for all $j\geq i-1$, and we conclude that $I_n^i\subseteq \ker(\varphi)$ and we get our desired map.
\end{proof}

To finish the proof of Theorem~\ref{thm:QuotientRing}, by Lemma~\ref{lem:CohMap} we must simply check that the cohomology ring and the quotient ring are both free of the same total rank. 

To do this, we invoke results of Rhoades~\cite{RhoadesSubspace} who studied quotient rings in the context of \emph{spanning subspace arrangements} in terms of ordered set partitions. Given $k\leq n$ two positive integers, $\mathcal{OP}_{n,k}$ denotes the set of \textbf{ordered set partitions} of $[n]$ with $k$ blocks, $\sigma = (B_1|B_2|\cdots |B_k)$ where $[n] = B_1\sqcup B_2\sqcup \cdots \sqcup B_k$. 

\begin{lemma}\label{lem:Ranks}
    The two rings $H^*(K^i)$ and $\mathbb{Z}[x_1,\dots, x_n]/I_n^i$ are free $\mathbb{Z}$-modules of the same rank.
\end{lemma}

\begin{proof}
    By Proposition~\ref{prop:Poincare-ij}, $H^*(K^i)$ is free and the associated Poincar\'e polynomial is
    \[
    \Poin(K^i;q) = [n-1]_q! [i-1]_q[n-i+1]_q.
    \]
    On the other hand, the quotient ring $\mathbb{Z}[x_1,\dots, x_n]/I_n^i$ has been studied by Rhoades in \cite{RhoadesSubspace}, where he showed it is free and found an explicit basis of the ring. In Rhoades' notation, $I_n^i = I_{(i-1,n+1-i),n-1}$. By \cite[Lemma 5.3]{RhoadesSubspace}, the quotient ring $\mathbb{Z}[x_1,\dots, x_n]/I_n^i$ has a basis indexed by the following collection of ordered set partitions
    \[
    \mathcal{OP}_{(i-1,n+1-i),n-1} = \left\{\sigma \in \mathcal{OP}_{n,n-1} \Bigm\vert 
    \begin{array}{l} 1,\dots, i-1\text{ are in distinct blocks of }\sigma\\
    i,\dots, n\text{ are in distinct blocks of }\sigma
    \end{array}\right\}.
    \]
    For $\sigma\in \mathcal{OP}_{(i-1,n+1-i),n-1}$, observe that all blocks $B_j$ of $\sigma$ must be singletons except for one block $B_m$ which has size 2. Thus, $B_m$ must contain one element from $\{1,\dots, i-1\}$ and one element from $i,\dots,n$. Thus,
    \[
    |\mathcal{OP}_{(i-1,n+1-i),n-1}| = (i-1)(n-i+1)(n-1)!,
    \]
    where $(i-1)(n-i+1)$ represents a choice of size 2 block, and $(n-1)!$ counts all permutations of the size 2 block with the $n-2$ singleton blocks. Therefore, 
    \[
    \mathrm{rk}(\mathbb{Z}[x_1,\dots, x_n]/I_n^i) = |\mathcal{OP}_{(i-1,n+1-i),n-1}| = (i-1)(n-i+1)(n-1)! = \mathrm{rk}(H^*(K^i)).
    \]
\end{proof}

\begin{proof}[Proof of Theorem~\ref{thm:QuotientRing}]
    By Lemma~\ref{lem:CohMap}, we have a surjective map $\mathbb{Z}[x_1,\dots, x_n]/I_n^i \twoheadrightarrow H^*(K^i)$. By Lemma~\ref{lem:Ranks}, both rings are free $\mathbb{Z}$-modules of the same total finite rank, hence the map must be an isomorphism.
\end{proof}

\section{Further directions}\label{sec:Future}

In this article, we focused on the case $\lambda=(1^{n-1})$ of the $\Delta$-Springer fiber $Y_{n,\lambda,s}$. When $\lambda$ has two columns, we expect that the irreducible components will not all be smooth. Indeed, this is already known to be true for the usual Springer fibers (when the size of $\lambda$ is equal to $n$). However, we expect that it may be a reasonable problem to characterize the smooth irreducible components of $Y_{n,\lambda,s}$ in all cases when $\lambda$ has two columns by extending techniques of Fresse and Melnikov~\cite{fresse-melnikov}.

\begin{problem}
    Generalize the characterization of smooth irreducible components of two-column Springer fibers given in \cite{fresse-melnikov} to the setting of $Y_{n,\lambda,s}$ for $\lambda_1\leq 2$.
\end{problem}

We also believe it would be interesting to generalize the cohomology presentation given in Section~\ref{sec:SingularCohomology} in various ways.

\begin{problem}
    Extend Theorem~\ref{thm:QuotientRing} to a presentation of each irreducible component $K^i$ of $Y_{n,(1^{n-1}),s}$ for general $s\geq n-1$ and to each irreducible component intersection $K^{i,j}$.
\end{problem}

\begin{problem}
    There is a $1$-parameter torus $\bC^\times$ that acts on $Y_{n,(1^{n-1}),s}$ and thus on each irreducible component $K^i$. Generalize the presentation in Theorem~\ref{thm:QuotientRing} to equivariant cohomology $H^*_T(K^i)$.
\end{problem}

\printbibliography

\end{document}